\documentclass[11pt]{article}

\usepackage{graphicx}
\usepackage{color}
\usepackage{amsmath}
\usepackage{amssymb}
\usepackage{amscd}
\usepackage{bbm}
\usepackage{bm}
\usepackage{enumerate}

\newcommand{\R}{\mathbb{R}}
\newcommand{\inr}[1]{\bigl< #1 \bigr>}

\newcommand{\N}{\mathbb{N}}
\newcommand{\Z}{\mathbb{Z}}
\newcommand{\E}{\mathbb{E}}
\renewcommand{\P}{\mathbb{P}}
\newcommand{\C}{\mathbb{C}}

\newcommand{\eps}{\epsilon}
\newcommand{\conv}{\mathop{\rm conv}}

\newcommand{\opn} {_{2\rightarrow 2}}
 \newcommand{\tr}{\operatorname{tr}}

\newtheorem{Theorem}{Theorem}[section]
\newtheorem{Lemma}[Theorem]{Lemma}

\newtheorem{Definition}[Theorem]{Definition}

\newtheorem{Proposition}[Theorem]{Proposition}
\newtheorem{Corollary}[Theorem]{Corollary}
\newtheorem{Remark}[Theorem]{Remark}

\numberwithin{equation}{section}

\def \proof {\noindent {\bf Proof.}\ \ }
\newcommand\proofof[1] {\noindent {\bf Proof of #1.}\ \ }

\def \endproof
{{\mbox{}\nolinebreak\hfill\rule{2mm}{2mm}\par\medbreak}}

\newcommand{\vct}[1]{\bm{#1}}
\newcommand{\mtx}[1]{\bm{#1}}

\begin{document}

\title{Suprema of Chaos Processes and the Restricted Isometry Property}

\author{Felix Krahmer\thanks{Institute for Numerical and Applied Mathematics, University of G{\"o}ttingen, Lotzestra\ss e 16-18, 37083 G{\"o}ttingen, Germany,
f.krahmer@math.uni-goettingen.de}, Shahar Mendelson\thanks{Department of Mathematics, Technion, Haifa, 32000, Israel, shahar@tx.technion.ac.il},
and Holger Rauhut\thanks{Hausdorff Center for Mathematics and Institute for Numerical Simulation,
University of Bonn, Endenicher Allee 60, 53115 Bonn, Germany, rauhut@hcm.uni-bonn.de}}

\date{June 30, 2012; revised June 20, 2013}

\maketitle
\abstract{We present a new bound for suprema of a special type of chaos processes indexed by a set of matrices, which is based on a chaining method.
As applications we show significantly improved estimates for the restricted isometry constants of partial random circulant matrices
and time-frequency structured random matrices. In both cases the required condition on the number $m$ of rows in terms of
the sparsity $s$ and the vector length $n$ is $m \gtrsim s \log^2 s \log^2 n$.
}

\bigskip


\section{Introduction and Main Results} \label{sec:introduction}

\subsection{Compressive Sensing}

Compressive sensing \cite{carota06,do06-2} is a method aimed at recovering sparse vectors from highly incomplete information
using efficient algorithms, see \cite{fora11,ra10} for expository articles. This discovery has recently triggered various applications
in signal and image processing.

To formulate the procedure, a vector $\vct{x} \in \C^n$ is called $s$-sparse if $\|\vct{x}\|_0 := |\{\ell: x_\ell \neq 0\}| \leq s$. Given a matrix
$\mtx{\Phi} \in \C^{m \times n}$, called the measurement matrix, the task is to reconstruct $\vct{x}$ from the linear measurements
\[
\vct{y} = \mtx{\Phi} \vct{x}.
\]
We are interested in the case $m \ll n$, so that this system is under-determined, and thus, without additional information it is impossible to reconstruct $\vct{x}$. On the other hand, if it is known a priori that $\vct{x}$ is $s$-sparse then the situation changes. Although the naive approach for reconstruction, namely, $\ell_0$-minimization,
\[
\min \|\vct{z}\|_0 \quad \mbox{ subject to } \mtx{\Phi} \vct{z} = \vct{y}
\]
is NP-hard in general, there are several tractable alternatives. A widely applied method is $\ell_1$-minimization \cite{chdosa99,do06-2,carota06}
\[
\min_{\vct{z}} \|\vct{z}\|_1 \quad \mbox{ subject to } \mtx{\Phi} \vct{z} = \vct{y},
\]
(where $\|\vct{z}\|_p$ denotes the usual $\ell_p$-norm) which is a convex optimization problem and may be solved efficiently.

The restricted isometry property streamlines the analysis of recovery algorithms. For a matrix $\mtx{\Phi} \in \C^{m \times n}$ and $s < n$, the restricted isometry constant $\delta_s$ is defined as the smallest number such that
\[
(1-\delta_s) \|\vct{x}\|_2^2 \leq \|\mtx{\Phi} \vct{x}\|_2^2 \leq (1+\delta_s) \|\vct{x}\|_2^2 \quad \mbox{ for all $s$-sparse } \vct{x}.
\]
One may show that under conditions of the form $\delta_{\kappa s} \leq \delta^*$ for some $\delta^* < 1$ and some appropriate small integer $\kappa$, a variety of recovery algorithms reconstruct every $s$-sparse $\vct{x}$ from $\vct{y} = \mtx{\Phi}\vct{x}$. Among these are $\ell_1$-minimization as mentioned above \cite{carota06-1,fo10,limo11}, orthogonal matching pursuit \cite{zh11},
CoSaMP \cite{netr08,fo10}, iterative hard thresholding \cite{blda09}
and hard thresholding pursuit \cite{fo10-2}.

Remarkably, all optimal measurement matrices known so far are random matrices.
For example, a Bernoulli random matrix $\mtx{\Phi} \in \R^{m \times n}$ has entries $\Phi_{jk} = \eps_{jk}/\sqrt{m}$, where the $\eps_{jk}$ are independent, symmetric $\{-1,1\}$-valued random variables. Its restricted isometry constant satisfies $\delta_s \leq \delta$ with probability at least $1-\eta$ provided that
\[
m \geq C \delta^{-2} (s \ln(en/s) + \ln(\eta^{-1})),
\]
where $C$ is an absolute constant \cite{cata06,mepato07,badadewa08}, see also Appendix~\ref{RIP:subgauss}.

In practice, structure is an additional requirement on the measurement matrix $\mtx{\Phi}$. Indeed, certain applications impose constraints
on the matrix and recovery algorithms can be accelerated when fast matrix vector multiplication routines are available
for $\mtx{\Phi}$. Unfortunately, a Bernoulli random matrix does not possess any structure. This motivates the study of  random matrices with more structure. Also, structured random matrix constructions usually involve a reduced degree of randomness. For example, partial random Fourier matrices $\mtx{\Phi} \in \C^{m \times n}$ arise as random row submatrices of the discrete Fourier matrix and their restricted isometry constants satisfy $\delta_s \leq \delta$ with high probability provided that
\[
m \geq C \delta^{-2} s \log^3 s \log n,
\]
see \cite{cata06,ruve08}.

This article provides a similar estimate for two further types of structured random matrices, namely partial random circulant matrices and
time-frequency structured random matrices. The key proof ingredients will be new estimates for suprema of chaos processes of a certain type.

\subsection{Partial random circulant matrices}
\label{Sec:PRCM}

Circulant matrices are connected to circular convolution, defined for two vectors $\vct{x},\vct{z} \in \C^n$ by
\[
(\vct{z} * \vct{x})_j := \sum_{k = 1}^n z_{j \ominus k} x_k,  \quad j =1,\hdots, n,
\]
where $j \ominus k = j - k \mod n$ is the cyclic subtraction.
The circulant matrix $\mtx{H}= \mtx{H}_{\vct{z}} \in \C^{n \times n}$ associated with $\vct{z}$ is given by
$\mtx{H} \vct{x} = \vct{z}*\vct{x}$ and has entries $H_{jk} = z_{j \ominus k}$.

We are interested in sparse recovery from subsampled convolutions with a random vector. Formally, let $\Omega \subset \{1,\hdots,n\}$ be an arbitrary (fixed)
set of cardinality $m$, and
denote by $\mtx{R}_\Omega : \C^n \to \C^m$ the operator that restricts a vector $\vct{x} \in \C^n$ to its entries in $\Omega$.
Let $\vct{\eps}=(\eps_i)_{i=1}^n$ be a Rademacher vector of length $n$, i.e., a random vector with independent entries distributed according to $\P(\eps_i=\pm 1)=\frac{1}{2}$. Then the associated partial random circulant matrix is given by
$\mtx{\Phi} = m^{-1/2} \mtx{R}_\Omega \mtx{H}_{\vct{\eps}}\in \R^{m \times n}$ and acts on vectors $\vct{x} \in \C^n$ via
\[
\mtx{\Phi} \vct{x} =  \frac{1}{\sqrt{m}} \mtx{R}_\Omega( \vct{\eps} * \vct{x} ).
\]
In other words, $\Phi$ is a circulant matrix generated by a Rademacher vector, where the rows outside $\Omega$ are removed.
Our first main result establishes the restricted isometry property of $\Phi$ in a near-optimal parameter regime:

\begin{Theorem}\label{thm:main:circRIP}
Let $\mtx{\Phi} \in \R^{m \times n}$ be a draw of a partial random circulant matrix generated by a Rademacher vector $\vct{\eps}$. If
\begin{equation}\label{m:circRIPcond}
m \geq c \delta^{-2} s\, (\log^2 s) (\log^2 n),
\end{equation}
then with probability at least $1-n^{-(\log n)(\log^2 s)}$, the restricted isometry constant of $\mtx{\Phi}$ satisfies $\delta_s \leq \delta$. The constant $c>0$ is universal.
\end{Theorem}

In Section~\ref{sec:RIPcirc}, we will prove a more general version of this theorem,
just requiring that the generating random variable is mean-zero, variance one, and subgaussian.
These results improve the best previously known
estimates for a partial random circulant matrix \cite{rarotr12}, namely that
$m \geq C_{\delta} (s \log n)^{3/2}$ is a sufficient condition
for achieving $\delta_s \leq \delta$ with high probability
(see also \cite{bahanora10} for an earlier work on this problem). In particular,
Theorem \ref{thm:main:circRIP} removes the exponent $3/2$ of the sparsity $s$,
which was already conjectured in \cite{rarotr12} to
be an artefact of the proof.

A related non-uniform recovery result is contained in \cite{ra09,ra10} where one considers the probability that a fixed $s$-sparse vector
is reconstructed via $\ell_1$-minimization using a draw of a partial random circulant matrix. The condition derived there is
$m \geq C s \log^2n$, which is slightly better than \eqref{m:circRIPcond}. However, the statement of Theorem \ref{thm:main:circRIP}
is considerably stronger because the restricted isometry property implies uniform and stable recovery of all $s$-sparse vectors via $\ell_1$-minimization and other recovery methods for a single matrix $\mtx{\Phi}$.

Note that in \cite{ro09-1}, the restricted isometry property has been established for partial random circulant matrices with
{\it random} sampling sets and random generators under the condition $m \geq C s \log^6 n$. In contrast, our result holds for an arbitrary fixed selection of a set $\Omega \subset \{1,\hdots,n\}$, which is important in applications since in many practical problems, it is natural or desired to consider structured sampling sets  such as $\Omega = \{L,2L,3L,\hdots, mL\}$ for some $L \in \N$; these sets are clearly far from being random.

Potential applications of compressive sensing with subsampled random convolutions
include system identification, radar and cameras with coded aperature.
We refer to \cite{bahanora10,ro09-1,rarotr12} for a discussion on these applications. 

Combining our result with the work \cite{krwa11} on the relation between the restricted isometry property and the Johnson-Lindenstrauss lemma
we also obtain an improved estimate for Johnson-Lindenstrauss embeddings
arising from partial random circulant matrices, see also \cite{hivy11,vy11} for earlier work in this direction.

\begin{Theorem} Fix $\eta,\delta \in (0,1)$, and consider a finite set $E \subset\mathbb{R}^n$ of cardinality $|E| = p$.  Choose
\[m \geq C_1\delta^{-2}\log(C_2 p)(\log\log(C_2 p))^2 (\log n)^2,\]
where the constants $C_1, C_2$ depend only on $\eta$. 
Let $\mtx{\Phi}\in\C^{m\times n}$ be a partial circulant matrix generated by a Rademacher vector $\vct{\epsilon}$.
Furthermore, let $\vct{\epsilon'} \in \mathbb{R}^n$ be a Rademacher vector 
independent of $\vct{\epsilon}$ and set $\mtx{D}_{\vct{\epsilon}'}$ to be the diagonal matrix with diagonal $\vct{\epsilon}'$.
Then with probability exceeding $1 - \eta$, for every $\vct{x}\in E$,
\begin{align*}
(1 - \delta) \| \vct{x} \|_2^2 \leq \| \mtx{\Phi D }_{\vct{\epsilon}'} \vct{x} \|_2^2 \leq (1 + \delta) \| \vct{x} \|_2^2.
\end{align*}
\end{Theorem}

\subsection{Time-frequency structured random matrices}
\label{Sec:Intro:Gabor}

The translation and modulation operators on $\C^m$ are defined by
$(\mtx{T} \vct{h})_j=h_{j\ominus 1}$
and $(\mtx{M}\vct{h})_j = e^{2\pi i j/m} h_j = \omega^j h_j$,
where $\omega=e^{2\pi i/m}$ and
$\ominus$ again denotes cyclic subtraction, this time modulo $m$. Observe that
\begin{equation}\label{eq:trans_mod2}
(\mtx{T}^k
\vct{h})_j=h_{j\ominus k}\quad\mbox{and}\quad (\mtx{M}^{\ell}\vct{h})_j = e^{2\pi i\ell j/n} h_j
= \omega^{\ell j} h_j.
\end{equation}
The time-frequency shifts are given by
\[
\mtx{\pi}(\lambda) = \mtx{M}^\ell \mtx{T}^k, \quad \lambda = (k,\ell) \in \Z_m^2=\{0,\hdots,m-1\}^2.
\]
For $\vct{h} \in \C^m \setminus \{0\}$ the system
$\{ \mtx{\pi}(\lambda) \vct{h}: \lambda \in  \Z_m^2\}$, is called a
Gabor system \cite{feluwe07,gr01,krpfra08}, and the $m \times m^2$
matrix $\mtx{\Psi}_{\vct{h}}$ whose columns are the vectors $\mtx{\pi}(\lambda) \vct{h}$ for $\lambda \in \Z_m^2$ is called a Gabor synthesis matrix,
\[
\mtx{\Psi}_{\vct{h}} = \big[\mtx{\pi}(\lambda) \vct{h}\big]_{\lambda \in Z_m^2}  \in \C^{m \times m^2}.
\]
Note that here the signal length $n$ is coupled to the embedding dimension $m$ via $n=m^2$ (so that $\log n =  2 \log m$ below).

Our second main result establishes the restricted isometry property for Gabor synthesis matrices generated by a random vector.
The following formulation again focuses on normalized Rademacher vectors, postponing a more general version of our results until Section~\ref{sec:RIPgabor}.

\begin{Theorem}\label{thm:main:Gabor} Let $\vct{\eps}$ be a Rademacher vector and consider 
the Gabor synthesis matrix $\mtx{\Psi}_{\vct{h}} \in \C^{m \times m^2}$ generated by $\vct{h} = \frac{1}{\sqrt{m}} \vct{\eps}$. If
\begin{equation}\label{m:GaborRIPcond}
m \geq c \delta^{-2} s\, 
(\log s)^2 
(\log m)^2, 
\end{equation}
then with probability at least $1-m^{-(\log m)\cdot(\log^2 s)}$, the restricted isometry constant of $\mtx{\Psi}_{\vct{h}}$ satisfies $\delta_s \leq \delta$.
\end{Theorem}
Again, Theorem \ref{thm:main:Gabor} improves the best previously known estimate from \cite{pfratr11}, in which the sufficient condition of $m \geq C s^{3/2} \log^3 m$ was derived. In particular, it implies the first uniform sparse recovery result with a linear scaling of the number of samples $m$ in the sparsity $s$ (up to $\log$-factors).

A non-uniform recovery result for Gabor synthesis matrices with Steinhaus generator (see Section~\ref{sec:prelim} for the definition) appears in \cite{pfra10},
where it was shown that a fixed $s$-sparse vector
is recovered from its image under a random draw of the $m \times m^2$ Gabor synthesis matrix via $\ell_1$-minimization with high probability provided that $m \geq C s \log m$. Again, the conclusion
of Theorem \ref{thm:main:Gabor} is stronger than this previous result in the sense that it implies
uniform and stable $s$-sparse recovery.
Further related material may be found in \cite{pfrata08,bahanosa10}.

Applications of random Gabor synthesis matrices include operator identification (channel estimation in wireless communications), radar and sonar \cite{bahanosa10,hest09,pfrata08}.

\subsection{Suprema of chaos processes} \label{sec:intro_chaos}
Both for partial random circulant matrices and for time-frequency structured random matrices generated by Rademacher vectors, the restricted isometry constants $\delta_s$ can be written as a random variable $X$ of the form
\begin{equation}\label{def:special:chaos}
X = \sup_{\mtx{A} \in {\cal A}} \left| \|\mtx{A} \vct{\epsilon}\|_2^2 - \E \|\mtx{A} \vct{\epsilon}\|_2^2 \right|,
\end{equation}
where ${\cal A}$ is a set of matrices and $\vct{\epsilon}$ is a Rademacher vector.
Due to the identity \eqref{def:posdef:chaos} below, $X$ is the supremum of a chaos process.

Our third main result -- the main ingredient of the proofs of Theorems~\ref{thm:main:circRIP} and \ref{thm:main:Gabor}, but also of independent interest -- provides expectation and deviation bounds for random vectors $X$ of this form in terms of two types of complexity parameters of the set of matrices $\cal A$.
The first one, denoted by $d_F({\cal A})$ and $d\opn({\cal A})$,
is the radius of $\cal A$ in  the Frobenius norm $\|\mtx{A}\|_F = \sqrt{\tr(\mtx{A}^*\mtx{A})}$ and
the operator norm $\|\mtx{A}\|\opn = \sup_{\|\vct{x}\|_2 \leq 1} \|\mtx{A}\vct{x}\|_2$, respectively. That is,  $d_F({\cal A})=\sup\limits_{\mtx{A}\in{\cal A}} \|\mtx{A}\|_F$ and $d\opn({\cal A})=\sup\limits_{\mtx{A}\in{\cal A}} \|\mtx{A}\|\opn$.
For the second one, Talagrand's functional $\gamma_2({\cal A},\|\cdot\|\opn)$, we refer to Definition~\ref{def:gamma-2} below for a precise description.

With these notions, our result reads as follows.

\begin{Theorem}\label{thm:main:chaos} Let ${\cal A} \subset \C^{m \times n}$ be a symmetric set of matrices, ${\cal A} = - {\cal A}$. Let $\vct{\epsilon}$ be a Rademacher
vector of length $n$. Then
\begin{equation}\label{main:chaos:estimate}
\E \sup\limits_{\mtx{A} \in {\cal A}} \left| \|\mtx{A} \vct{\epsilon} \|_2^2 - \E \|\mtx{A} \vct{\epsilon}\|_2^2 \right| \leq C_1 \left( d_F({\cal A}) \gamma_2({\cal A},\|\cdot\|\opn) +  \gamma_2({\cal A},\|\cdot\|\opn)^2\right)=:C_1 E. 
\end{equation}
Furthermore, for $t>0$,
\begin{equation}\label{prob:main:bound}
\P\left( \sup_{\mtx{A} \in {\cal A}} \left| \|\mtx{A} \vct{\eps}\|_2^2 - \E \|\mtx{A} \vct{\eps}\|_2^2 \right| \geq C_2 E 
+ t \right) \leq 2 \exp\left(-C_3 \min\left\{\frac{t^2}{V^2},\frac{t}{U}\right\}\right),
\end{equation}
where
\begin{equation*}
V=d_{2 \to 2}\left({\cal A})(\gamma_2({\cal A},\| \cdot \|_{2 \to 2}\right)+d_F({\cal A})) \qquad {\rm and} \qquad U=d_{2 \to 2}^2({\cal A}).
\end{equation*}
The constants $C_1,C_2, C_3 > 0$ are universal.
\end{Theorem}
The symmetry assumption ${\cal A} = - {\cal A}$ was made for the sake of simplicity.
The more general Theorem \ref{thm:main-tail}
below does not use this assumption but requires an additional term on the right hand side of the estimate. Furthermore, Theorem \ref{thm:main-tail} will actually be stated under more general conditions on the generating random vector.

Let us relate our new bound to previous estimates. By expanding the $\ell_2$-norms we can rewrite $X$ in \eqref{def:special:chaos} as
\begin{equation}\label{def:posdef:chaos}
X = \sup\limits_{\mtx{A} \in {\cal A}} \left| \sum_{j\neq k} \epsilon_j \epsilon_k (\mtx{A}^* \mtx{A})_{j,k}\right|,
\end{equation}
which is a homogeneous chaos processes of order $2$ indexed
by the positive semidefinite matrices $\mtx{A}^* \mtx{A}$.
Talagrand \cite{ta05-2} considers general homogeneous
chaos process of the form
\[
Y = \sup\limits_{\mtx{B} \in {\cal B}} \left| \sum_{j \neq k} \epsilon_j \epsilon_k B_{j,k}\right|,
\]
where ${\cal B} \subset \C^{n \times n}$ is a set of (not necessarily positive semidefinite) matrices. He derives the bound
\begin{equation}\label{bound:Talagrand}
\E Y \leq C_1 \gamma_2({\cal B}, \|\cdot\|_F) + C_2 \gamma_1({\cal B}, \|\cdot\|\opn)
\end{equation}
(see Section~\ref{sec:prelim} for the definition of the $\gamma_\alpha$-functional).
This estimate was an essential component in the proofs of the previous bounds for the restricted isometry constants of partial random circulant matrices \cite{rarotr12} and
of random Gabor synthesis matrices \cite{pfratr11}. In fact, the appearance of the $\gamma_1$-functional leads to the non-optimal exponent $3/2$
in the sparsity $s$ in the estimate of the required number $m$ of samples.
In contrast, as our bound for the chaos at hand does not involve the $\gamma_1$-functional
but only the $\gamma_2$-functional, 
this issue does not arise here.

\begin{remark}
 The benchmark problems of estimating the singular values and the restricted isometry constant of a Bernoulli matrix (with independent $\pm 1$ entries) can also be recast as a supremum of chaos processes of the form \eqref{def:special:chaos}. The bounds 
 resulting from Theorem~\ref{thm:main:chaos} are then optimal up to a constant factor. We refer to Appendix~\ref{RIP:subgauss} 
 for this derivation.
Again, we are not aware of a way to deduce such bounds from \eqref{bound:Talagrand}.
\end{remark}
\subsection*{Acknowledgements}

H.R. and F.K. acknowledge support by the Hausdorff Center for Mathematics.
H.R. is funded by the European Research council by the Starting Grant StG 258926.
Parts of this research were developed during a stay of H.R. and F.K. at the Institute for Mathematics and Its Applications, University of Minnesota. Both
are grateful for the support and the stimulating research environment. {S.M. acknowledges the support of the Centre for Mathematics and its Applications,
The Australian National University, Canberra,
Australia. Additional support to S.M. was given by an Australian Research Council Discovery grant DP0986563, the European Community's Seventh Framework Programme (FP7/2007-2013) under ERC grant agreement 203134, and the Israel Science Foundation grant 900/10.}

\section{Preliminaries}\label{sec:prelim}
\subsection{Chaining} \label{sec:chaining}
The following definition is due to Talagrand \cite{ta05-2} and forms the core of the generic chaining methodology.
\begin{Definition} \label{def:gamma-2}
For a metric space $(T,d)$, an {\it admissible sequence} of $T$ is a
collection of subsets of $T$, $\{T_r : r \geq 0\}$, such that for
every $s \geq 1$, $|T_r| \leq 2^{2^r}$ and $|T_0|=1$. For $\beta \geq 1$,
define the $\gamma_\beta$ functional by
\begin{equation*}
\gamma_\beta(T,d) =\inf \sup_{t \in T} \sum_{r=0}^\infty
2^{r/\beta}d(t,T_r),
\end{equation*}
where the infimum is taken with respect to all admissible sequences
of $T$.
\end{Definition}
Recall that for a metric space $(T,d)$ and $u>0$, the covering number $N(T,d,u)$ is the minimal number of open balls of radius $u$ in $(T,d)$ needed to cover $T$.
The $\gamma_\alpha$-functionals can be bounded in terms of such covering numbers by the well-known Dudley integral (see, e.g., \cite{ta05-2}).
A more specific formulation for the $\gamma_2$-functional of a set of matrices $\cal A$ endowed with the operator norm, the scenario which we will focus on in this article, is 
\begin{equation} \label{eq:Dudley}
\gamma_2({\cal A},\|\cdot\|\opn) \leq c \int_0^{d\opn({\cal{A}})} \sqrt{\log N({\cal A},\|\cdot\|\opn,u)} du
\end{equation}
This type of entropy integral was introduced by Dudley \cite{du67}
to bound the supremum of Gaussian processes, and was extended by Pisier \cite{pi83-2} as a way of bounding processes that satisfy different decay properties.

When considered for a set $T \subset L_2$, $\gamma_2$ has close connections with
properties of the canonical Gaussian process indexed by $T$; we
refer the reader to \cite{du99,ta05-2} for detailed
expositions on these connections. One can show that under mild measurability
assumptions, if $\{G_t: t \in T\}$ is a centered Gaussian process indexed by a set $T$, then
\begin{equation}
c_1 \gamma_2(T,d) \leq \E \sup_{t \in T} G_t \leq c_2 \gamma_2(T,d), \label{eq:majme}
\end{equation}
where $c_1$ and  $c_2$ are  absolute constants, and for every $s,t
\in T$, $d^2 (s,t) = \E|G_s-G_t|^2$.  The upper bound is due to
Fernique \cite{fe75-1} and the lower bound is Talagrand's majorizing
measures theorem \cite{ta87,ta05-2}.

\subsection{Subgaussian random vectors} \label{sec:vec}

In this section, we will discuss different classes of random 
vectors that are needed in the formulation of the main results in a more general framework. We refer
to \cite{ve12-1} for further background material.
In the following definition, $S^{n-1}$ denotes the unit sphere in $\R^n$ (resp. in $\C^n$).

\begin{Definition} 
A mean-zero random vector $X$ on $\C^n$
is called {\bf isotropic} if for every $\theta \in S^{n-1}$, $\E|\inr{X,\theta}|^2 =1$.
A random vector $X$ is called {\bf $L$-subgaussian} if it is isotropic and $\P(|\inr{X,\theta}| \geq t) \leq 2\exp(-t^2/2L^2)$  for every $\theta \in S^{n-1}$ and any $t>0$.
\end{Definition}

It is well known that, up to an absolute constant, the tail estimates in the definition of a subgaussian random vector
are equivalent to the moment characterization
\begin{equation}\label{eq:subg_mom}
\sup\limits_{\theta\in S^{n-1}} \left(\E |\inr{X,\theta}|^p\right)^{1/p} \leq \sqrt{p} L.
\end{equation}

Assume that a random vector $\vct{\xi}$ has independent coordinates $\xi_i$, each of which is an $L$-subgaussian random variable of mean zero and variance one.
One may verify by a direct computation that $\vct{\xi}$ is $L$-subgaussian.
Rademacher vectors, standard Gaussian vectors, (that is, random vectors with independent 
normally distributed entries of mean zero and variance one), as well as Steinhaus vectors (that is, random vectors with independent
entries that are uniformly distributed on $\{z\in\C: |z|=1\}$), are examples of isotropic,
$L$-subgaussian random vectors for an absolute constant $L$.

Our derivation of the probability bound \eqref{prob:main:bound} 
requires the following well-known estimate relating strong and weak moments.
For convenience, a proof based on chaining and the majorizing measures theorem is provided in the appendix.
Note, however, that the next theorem will not be required for the estimate \eqref{main:chaos:estimate} 
of the expectation, and in the Rademacher
case different techniques may be applied, see also Remark \ref{remRademacher}.

\begin{Theorem} \label{thm:subgaussian}
Let $\vct{x}_1,\hdots,\vct{x}_n \in \C^N$ and $T\subset \C^N$. If $\vct{\xi}$ is an isotropic, $L$-subgaussian random vector 
and $\vct{Y}=\sum_{j=1}^n \xi_j \vct{x}_j$, then for every $p \geq 1$,
\begin{equation}
\left(\E  \sup_{\vct{t} \in T}|\langle \vct{t}, \vct{Y} \rangle |^p\right)^{1/p} \leq c \left(\E  \sup_{\vct{t} \in T}|\langle \vct{t}, \vct{G}\rangle | + \sup_{\vct{t} \in T} (\E |\langle \vct{t},\vct{Y}\rangle|^p)^{1/p} \right), \label{eq:subgaussian}
\end{equation}
where  $c$ is a constant which depends only on $L$ and $\vct{G}=\sum_{j=1}^N g_j \vct{x}_j$ for $g_1,\hdots,g_N$ independent standard normal random variables.
\end{Theorem}

Note that if $\|\cdot\|$ is some norm on $\C^N$ and $B_*$ is the unit ball in the dual norm of $\|\cdot\|$ then the above
theorem implies that
\[
\left(\E \|\vct{Y}\|^p\right)^{1/p} \leq c \left( \E \|\vct{G}\| + \sup_{\vct{t} \in B_*} \left(\E |\langle \vct{t}, \vct{Y}\rangle|^p\right)^{1/p} \right).
\]

In the remainder of this article, we will state and prove generalizations of
our main results Theorem \ref{thm:main:circRIP}, Theorem \ref{thm:main:Gabor},
and Theorem~\ref{thm:main:chaos}  to arbitrary isotropic vectors, whose
coordinates are independent $L$-subgaussian random variables.
Since a Rademacher vector has all these properties, the above formulations of our results will directly follow. 

\subsection{Further probabilistic tools}

The following decoupling inequality is a slight variation of a result found for instance in \cite{gide99}, see also \cite{botz87,ra10}.

\begin{Theorem}\label{thm:decouple}
 Let $\vct{\xi} = (\xi_1,\hdots, \xi_n)$ be a sequence of independent, centered 
random variables, and let $F$ be a convex function. If ${\cal B}$ is a collection of matrices and $\vct{\xi}'$ is an independent copy of $\vct{\xi}$, then
\begin{equation}\label{ineq:decouple}
\E \sup_{\mtx{B} \in {\cal B}} F\left(\sum_{\substack{j,k = 1\\ j \neq k}}^n \xi_j \xi_k \mtx{B}_{j,k} \right)
\leq \E \sup_{\mtx{B} \in {\cal B}} F\left(4\sum_{j,k=1}^n \xi_j \xi'_k \mtx{B}_{j,k}\right).
\end{equation}
\end{Theorem}

In order to handle a certain diagonal term in the general subgaussian case, 
we require also a slightly stronger decoupling inequality which is valid in the Gaussian case and
follows from specifying results from \cite[Section 2]{argi93} to an order 2 Gaussian chaos.

\begin{Theorem}\label{thm:decouple:Gaussian}
There exists an absolute constant $C$ such that the following holds for all $p\geq 1$. Let $\vct{g} = (g_1,\hdots, g_n)$ be a sequence of independent standard normal random variables. If ${\cal B}$ is a collection of Hermitian matrices and $\vct{g}'$ is an independent copy of $\vct{g}$,
then
\begin{equation}\label{ineq:decouple:Gaussian}
\E \sup_{\mtx{B} \in {\cal B}} \big| \sum_{\substack{j,k = 1\\ j \neq k}}^n g_j g_k \mtx{B}_{j,k} + \sum_{j=1}^n (g_j^2 -1) \mtx{B}_{j,j} \big|^p 
\leq C^p \E \sup_{\mtx{B} \in {\cal B}} \big| \sum_{j,k=1}^n g_j g'_k \mtx{B}_{j,k}\big|^p.
\end{equation}
\end{Theorem}

Since some steps in our estimates are formulated in terms of moments, the
transition to a tail bound can be established by the following standard estimate, which is a straightforward consequence
of Markov's inequality.

\begin{Proposition}\label{prop:mom2tail}
Suppose $Z$ is a random variable satisfying
\[
(\E |Z|^p)^{1/p} \leq \alpha + \beta \sqrt{p}+ \gamma p \quad \mbox{ for all } p \geq p_0
\]
for some $\alpha,\beta,\gamma,p_0 > 0$.
Then, for $u \geq p_0$,
\[
\P\big(|Z| \geq e(\alpha + \beta \sqrt{u} + \gamma u)\big) \leq e^{-u}\;.
\]
\end{Proposition}

\subsection{Notation}
Absolute constants will be denoted by $c_1,c_2,\hdots$; their value may change from line to line.
We write $A \lesssim B$ if there is an absolute constant $c_1$ for which $A \leq c_1 B$.
$A \sim B$ means that $c_1 A \leq B \leq c_2 A$ for absolute constants $c_1$ and $c_2$.
If the constants depend on some parameter $r$ we will write $A \lesssim_r B$ or $A \sim_r B$.

The $L_p$-norm of a random variable, or its $p$-th moment, is given by $\left\|X\right\|_{L_p} = \left(\E |X|^p\right)^{1/p}$. For a random variable $X$ independent from all other random variables which appear, we denote the expectation and probability conditional on all variables except $X$ by $\E_{X}$ and $\P_{X}$, respectively.
The canonical unit vectors in $\C^n$ are denoted by $\vct{e}_j$ and $B_2^n$ is the unit $\ell_2$-ball in $\C^n$.

Finally, we introduce shorthand notations for some quantities that we will study.
To that end, let ${\cal A}$ be a set of matrices on $\R^n$ or on $\C^n$ and set a random vector $\vct{\xi}=(\xi_i)_{i=1}^n$. For a given matrix $\mtx{A}$, denote its $j$-th column by $\vct{A}^j$ and set

 \begin{align*}
  N_{\cal A}(\vct{\xi}) &:= \sup_{\mtx{A} \in {\cal A}} \|\vct{A} \vct{\xi}\|_2,\qquad
 & B_{\cal A}(\vct{\xi}) &:= \sup_{\mtx{A} \in {\cal A}}\left| \sum_{\substack{j,k=1\\j\neq k}}^n \xi_j \overline{\xi_k}
 \langle \vct{A}^j, \vct{A}^k\rangle \right|,\\
D_{\cal A}(\vct{\xi}) &:= \sup_{\mtx{A} \in {\cal A}} \left|\sum_{j=1}^n \left(|\xi_j|^2 -  \E |\xi_j|^2\right)\|\vct{A}^j\|_2^2\right|, \qquad
\text{ and } \quad & C_{\cal A}(\vct{\xi})&:=\sup_{\mtx{A} \in {\cal A}} \left| \|\mtx{A} \vct{\xi}\|_2^2 - \E \|\mtx{A} \vct{\xi}\|_2^2 \right|.
\end{align*}

\section{Chaos processes} \label{sec:chaos}

We are now well-equipped to prove the following generalized version of Theorem~\ref{thm:main:chaos}.
\begin{Theorem} \label{thm:main-tail}
Let ${\cal A}$ be a set of matrices, and let $\vct{\xi}$ be a random vector whose entries $\xi_j$ are independent, mean-zero, variance $1$, and $L$-subgaussian random variables. Set

\begin{align*}
E&=\gamma_2({\cal A},\| \cdot \|_{2 \to 2}) \left( \gamma_2({\cal A},\| \cdot \|_{2 \to 2})
+d_F({\cal A})\right) + d_F({\cal A}) d\opn({\cal A}),\notag\\
V&=d_{2 \to 2}({\cal A})(\gamma_2({\cal A},\| \cdot \|_{2 \to 2})+d_F({\cal A})), \ \ {\rm and} \ \ U=d_{2 \to 2}^2({\cal A}).
\end{align*}
Then, for $t>0$,
\begin{equation*}
\P \left( C_{\cal A} (\vct{\xi})\geq c_1 E + t \right) \leq 2 \exp\left(-c_2 \min\left\{\frac{t^2}{V^2},\frac{t}{U}\right\}\right).
\end{equation*}
The constants $c_1, c_2$ depend only on $L$. 
\end{Theorem}

The proof is based on estimating the moments of the random variables $N_{\cal A}$ and $C_{\cal A}$, followed by applying Proposition~\ref{prop:mom2tail}. The first step is a bound on the moments of a decoupled version of $N_{\cal A}$.

\begin{Lemma} \label{lemma:chaining}
Let ${\cal A}$ be a set of matrices, let $\vct{\xi}=(\xi_i)_{i=1}^n$ be an $L$-subgaussian random vector, and let $\vct{\xi}'$ be an independent copy of $\vct{\xi}$. Then for every $p \geq 1$,
\begin{equation*}
\left\|\sup_{\mtx{A} \in {\cal A}} \inr{\mtx{A}\vct{\xi},\mtx{A}\vct{\xi}'}\right\|_{L_p} \lesssim_L \gamma_2({\cal A}, \| \cdot \|_{2 \to 2})\|N_{\cal A}(\vct{\xi})\|_{L_p} + \sup_{\mtx{A} \in {\cal A}} \| \inr{\mtx{A} \vct{\xi},\mtx{A} \vct{\xi}'}\|_{L_p}.
\end{equation*}
\end{Lemma}

\proof
The proof is based on a chaining argument. Since the space involved is finite dimensional, one may assume without loss of generality that ${\cal A}$ is finite. 
Let $(T_r)$ be an admissible sequence of $\cal{A}$ for which the minimum in the definition of $\gamma_2({\cal A},\|\cdot\|_{2 \to 2})$ is attained.
Let $\pi_r \mtx{A} =
{\rm argmin}_{\mtx{B}\in T_r} \|\mtx{B}-\mtx{A}\|_{2 \to 2}$ and set $\Delta_r \mtx{A} = \pi_{r} \mtx{A} - \pi_{r-1} \mtx{A}$.  Since ${\cal A}$ is
finite, there is some $r_0$ for which $|{\cal A}| \leq
2^{2^{r_0}}$. Given $p \geq 1$, let $\ell$ be the largest integer for which $2^\ell \leq p$, and we may assume that $\ell < r_0$ as the modifications needed when $\ell \geq r_0$ are minimal.

Note that for every $\mtx{A} \in {\cal A}$,
\begin{align}
& | \inr{\mtx{A}\vct{\xi}, \mtx{A}\vct{\xi}'}
- \inr{(\pi_{\ell}\mtx{A})\vct{\xi}, (\pi_{\ell}\mtx{A})\vct{\xi}'}|
\nonumber\\
\leq & \sum_{r=\ell}^{r_0-1} |\inr{\left(\Delta_{r+1}\mtx{A}\right) \vct{\xi},
(\pi_{r+1}\mtx{A}) \vct{\xi}'}|
+\sum_{r=\ell}^{r_0-1} |\inr{(\pi_{r}\mtx{A}) \vct{\xi},
\left(\Delta_{r+1}\mtx{A}\right) \vct{\xi}'}| =: S_1 + S_2. \label{eq:defS1S2}
\end{align}
Furthermore, conditionally on $\vct{\xi}'$,
\begin{equation*}
\inr{(\Delta_{r+1} \mtx{A}) \vct{\xi}, (\pi_{r+1} \mtx{A})\vct{\xi}'} =  \inr{\vct{\xi}, (\Delta_{r+1} \mtx{A})^*(\pi_r \mtx{A})\vct{\xi}'}
\end{equation*}
is a subgaussian random variable, as for every $u>0$,
\begin{equation}\label{eq:chstep}
\P_{\vct{\xi}} \left( |\inr{\vct{\xi}, (\Delta_{r+1} \mtx{A})^*(\pi_{r+1} \mtx{A})\vct{\xi}'}| \geq uL \|(\Delta_{r+1} \mtx{A})^*(\pi_{r+1} \mtx{A})\vct{\xi}'\|_2 \right) \leq 2\exp(-u^2/2).
\end{equation}
Recall that $|\{\pi_r \mtx{A} : \mtx{A} \in {\cal A}\}|=|T_r|\leq2^{2^r}$, so there are at most $2^{2^{r+2}}$ possible values that $\Delta_r (\mtx{A})^* \pi_{r+1}(\mtx{A})$
 can assume in \eqref{eq:chstep}. Therefore, via a union bound over all these choices and on all the levels $\ell < r \leq r_0$ (cf.\ \eqref{chain:estimate} below), it is evident that there are constants $c_1,c_2>0$ for which if $t \geq c_1$ then with $\vct{\xi}$-probability at least $1-2\exp(-c_22^{\ell}t^2)$ one has for every $\ell < r \leq r_0$ and every $\mtx{A} \in {\cal A}$
\begin{equation}\label{eq:event}
|\inr{(\Delta_{r+1}\mtx{A}) \vct{\xi},
(\pi_{r+1}\mtx{A}) \vct{\xi}'}| \leq t 2^{r/2}  \|(\Delta_{r+1}\mtx{A})^*(\pi_{r+1}\mtx{A})\vct{\xi}'\|_2.
\end{equation}

Let ${\cal E}_t(\vct{\xi}')$ be the event for which \eqref{eq:event} holds for all the possible choices of $r$ and $\mtx{A}$ as above. Since
\begin{equation*}
\|\left(\Delta_{r+1}\mtx{A}\right)^*(\pi_{r+1}\mtx{A})\vct{\xi}'\|_2 \leq \|\Delta_{r+1}\mtx{A}\|_{2 \to 2} \sup_{\mtx{A} \in {\cal A}} \|A \vct{\xi}'\|_2
= \|\Delta_{r+1}\mtx{A}\|_{2 \to 2} N_{\cal A}(\vct{\xi}'),
\end{equation*}
one has on ${\cal E}_t(\vct{\xi}')$
\begin{align*}
S_1 =\sum_{r=\ell}^{r_0-1} |\inr{\left(\Delta_{r+1}\mtx{A}\right) \vct{\xi},
(\pi_{r+1}\mtx{A}) \vct{\xi}'} |
\leq &  \sum_{r=\ell}^{r_0-1} t 2^{r/2} \|\Delta_{r+1}\mtx{A}\|_{2 \to 2} N_{\cal A}(\vct{\xi}')
\\
\leq & t \gamma_2({\cal A}, \| \cdot \|_{2 \to 2}) N_{\cal A}(\vct{\xi}').
\end{align*}
We will now estimate 
\begin{equation}
 \|S_1\|_{L_p}^p = \E_{\vct{\xi}} \E_{\vct{\xi}'} S_1^p = \E_{\vct{\xi}'} \int_0^\infty pt^{p-1} \P_{\vct{\xi}} (S_1>t | \xi') dt.
\end{equation}
Setting $W(\vct{\xi}')=\gamma_2({\cal A}, \| \cdot \|_{2 \to 2}) N_{\cal A}(\vct{\xi}')$, observe that
\begin{align*}
& \int_0^\infty pt^{p-1} \P_{\vct{\xi}} (S_1>t | \vct{\xi}') dt \leq c_3^pW(\vct{\xi}')^p + \int_{c_3W(\vct{\xi}')}^\infty pt^{p-1} \P_{\vct{\xi}} (S_1>t | \vct{\xi}') dt
\\
\leq & c_3^pW(\vct{\xi}')^p + W(\vct{\xi}')^p \int_{c_3}^\infty pu^{p-1} \P_{\vct{\xi}} (S_1>uW(\vct{\xi}') | \vct{\xi}') du \leq c_4^pW(\vct{\xi}')^p,
\end{align*}
where $c_3\geq c_1$ and $c_4$ are constants that depends only on $L$. Indeed, for $u \geq c_1$,
\[
\P_{\vct{\xi}} (S_1>uW(\vct{\xi}') |\vct{\xi}') \leq \P_{\vct{\xi}}({\cal E}_u^c(\vct{\xi}')|\vct{\xi}') \leq 2\exp(-c_2u^22^\ell) \leq 2\exp(-c_2u^2p/2).
\]
Repeating this argument for $S_2=\sum_{r=\ell}^{r_0-1} |\inr{(\pi_{r}\mtx{A}) \vct{\xi},
\left(\Delta_{r+1}\mtx{A}\right) \vct{\xi}'}|$, it follows that
\begin{equation}
\|S_1+S_2\|_{L_p} \leq c_5(L) \gamma_2({\cal A},\| \cdot \|_{2 \to 2}) \|N_{\cal A}(\vct{\xi})\|_{L_p}. \label{eq:chain_dist}
\end{equation}
Finally, since $|\{\pi_\ell \mtx{A} : \mtx{A} \in {\cal A}\}|\leq 2^{2^{\ell}} \leq \exp(p)$, we conclude
\begin{align*}
 \E \sup_{\mtx{A} \in {\cal A}} |\inr{(\pi_\ell \mtx{A}) \vct{\xi}, (\pi_\ell \mtx{A}) \vct{\xi}'}|^p \leq & \sum_{ \widetilde{\mtx{A}} \in T_\ell}
\E  |\inr{\widetilde{\mtx{A}} \vct{\xi}, \widetilde{\mtx{A}} \vct{\xi}'}|^p
\\
\leq & \exp(p) \sup_{\mtx{A} \in {\cal A}} \E |\inr{\mtx{A} \vct{\xi}, \mtx{A} \vct{\xi}'}|^p,
\end{align*}
and thus
\begin{equation}
\|\sup_{\mtx{A} \in {\cal A}} |\inr{(\pi_\ell \mtx{A}) \vct{\xi}, (\pi_\ell \mtx{A}) \vct{\xi}'}|\, \|_{L_p} \leq e\sup_{\mtx{A} \in {\cal A}} \|\inr{\mtx{A} \vct{\xi}, \mtx{A} \vct{\xi}'}\|_{L_p}. \label{eq:chain_appr}
\end{equation}
Combining \eqref{eq:defS1S2}, \eqref{eq:chain_dist}, and \eqref{eq:chain_appr} via the triangle inequality completes the proof.
\endproof

With these preliminary results at hand, we can now proceed to establish moment bounds for the quantities in questions. To illustrate the main idea, we start by considering the first moment in the Rademacher case. 
\begin{Corollary}\label{cor:rademacher}
Let ${\cal A} \subset \C^{m \times n}$ be a set of matrices and let $\vct{\epsilon}$ be a Rademacher
vector of length $n$. Then
\begin{equation*}
\E \sup\limits_{\mtx{A} \in {\cal A}} \left| \|\mtx{A} \vct{\epsilon} \|_2^2 - \E \|\mtx{A} \vct{\epsilon}\|_2^2 \right| 
\leq C_1 \left(  \gamma_2({\cal A},\|\cdot\|\opn)^2 + d_F({\cal A}) \gamma_2({\cal A},\|\cdot\|\opn) + d_F({\cal A}) d_{\opn}({\cal A})\right).
\end{equation*}
\end{Corollary} 
\begin{proof} By the decoupling inequality of Theorem \ref{thm:decouple}, we have
\begin{align}
\E C_{\cal A}(\vct{\epsilon}) & = \E \sup_{\mtx{A} \in {\cal A}} \left| \|\mtx{A} \vct{\epsilon} \|_2^2 - \E  \|\mtx{A} \vct{\epsilon}\|_2^2 \right| 
= \E \sup_{\mtx{A} \in {\cal A}} \left| \sum_{j \neq k} \epsilon_j \epsilon_k (\mtx{A}^* \mtx{A})_{j,k}\right|\notag\\
& \leq 4 \E \sup_{\mtx{A} \in {\cal A}} \left| \sum_{j, k} \epsilon_j \epsilon_k' (\mtx{A}^* \mtx{A})_{j,k}\right| 
 \lesssim \gamma_2({\cal A}, \|\cdot\|_{\opn}) \E N_{{\cal A}}(\vct{\epsilon}) 
+ \sup_{\mtx{A} \in {\cal A}} \E |\langle \mtx{A} \vct{\epsilon}, \mtx{A} \vct{\epsilon}'\rangle |.\label{bound1CA}
\end{align}
Since 
\[
\E |\langle \mtx{A} \vct{\epsilon}, \mtx{A} \vct{\epsilon}'\rangle| \leq \sqrt{ \E  |\langle \mtx{A} \vct{\epsilon}, \mtx{A} \vct{\epsilon}'\rangle|^2}
= \sqrt{\E \|\mtx{A}^* \mtx{A} \vct{\epsilon}\|_2^2} = \|\mtx{A}^*\mtx{A}\|_F \leq \|\mtx{A}\|_{\opn} \|\mtx{A}\|_F 
\]
we have
\begin{align}\label{dfdopn}
 \sup_{\mtx{A} \in {\cal A}} \E |\langle \mtx{A} \vct{\epsilon}, \mtx{A} \vct{\epsilon}'\rangle | 
 \leq d_{2 \to 2}({\cal A}) d_F({\cal A}) \leq d_F({\cal A})^2.
\end{align}
We conclude that
\[
(\E N_{\cal A}(\vct{\epsilon}))^2 \leq \E N_{\cal A}(\vct{\epsilon})^2 \leq \E C_{\cal A}(\vct{\epsilon}) + d_F({\cal A}) \lesssim
\gamma_2({\cal A}, \|\cdot\|_{\opn}) \E N_{{\cal A}}(\vct{\epsilon}) + d_F({\cal A})^2, 
\]
so that
\[
\E N_{\cal A}(\vct{\epsilon}) \lesssim \gamma_2({\cal A}, \|\cdot\|_{\opn}) + d_F({\cal A}).
\]
Plugging into \eqref{bound1CA} and using the first inequality in \eqref{dfdopn} yields the claim.
\end{proof}
\begin{Remark} 
Noting that the symmetry assumption ${\cal A} = -{\cal A}$ ensures that $d\opn({\cal A})\leq \gamma_2({\cal{A}}, \|\cdot \|\opn)$, Corollary~\ref{cor:rademacher} directly implies the expectation bound given in Theorem~\ref{thm:main:chaos}.
The tail bound -- and hence the remaining part of the proof of  Theorem~\ref{thm:main:chaos} -- can be deduced using the concentration inequality
in \cite[Theorem 17]{boluma03},
see also \cite{ta96-2}. Note that this proof does not use Theorem~\ref{thm:subgaussian} above and hence does not rely on the majorizing measures theorem. The proof, however, is specific to the case of Rademacher random vectors.
\end{Remark}
\begin{Theorem} \label{thm:moment-main}
Let $L \geq 1$ and $\vct{\xi}=(\xi_j)_{j=1}^n$, where $\xi_j$, $j=1, \dots, n$, are independent mean-zero, variance one, $L$-subgaussian random variables, and let ${\cal A}$ be a class of matrices. Then for every $p \geq 1$,
\begin{enumerate}[{(}a{)}]
 \item
\begin{equation*}
\left\| N_{\cal A} (\vct{\xi})\right\|_{L_p} \lesssim_L  \gamma_2({\cal A}, \| \cdot \|_{2 \to 2}) + d_F({\cal A})+ \sqrt{p}d_{2 \to 2} ({\cal A}), \label{eq:NAbound}
\end{equation*}
\item
\begin{align*}
\left\| C_{\cal A}(\vct{\xi})\right\|_{L_p} \lesssim_L & \gamma_2({\cal A}, \| \cdot \|_{2 \to 2})\left(\gamma_2({\cal A}, \| \cdot \|_{2 \to 2})+d_F({\cal A}) \right)
\\
&  + \sqrt{p} d_{2 \to 2}({\cal A})\left( \gamma_2({\cal A}, \| \cdot \|_{2 \to 2}) +d_{F}({\cal A})\right)+ pd^2_{2 \to 2}({\cal A}).
\end{align*}
\end{enumerate}
\end{Theorem}

Two important steps of the proof are established by the following lemmas. The lemmas use the same notation as Theorem~\ref{thm:moment-main}. The first lemma bounds a decoupled variant of the above quantities.
\begin{Lemma}\label{lem:decbound}
  If $\vct{\xi}'$ is an independent copy of $\vct{\xi}$, then
\begin{equation} \label{eq:lemdecbound}
\sup_{\mtx{A} \in {\cal A}} \| \inr{\mtx{A} \vct{\xi},\mtx{A} \vct{\xi}'}\|_{L_p}
\lesssim_L \sqrt{p} d_F({\cal A}) d_{2 \to 2}({\cal A}) + p d_{2 \to 2}^2({\cal A}).
\end{equation}
\end{Lemma}

\proof
Fix $\mtx{A} \in {\cal A}$ and set $S=\{\mtx{A}^* \mtx{A} {\vct{x}} : {\vct{x}} \in B_2^n\}$. Since the random vector $\vct{\xi}$ is $L$-subgaussian, the random variable  $\inr{\vct{\xi},\mtx{A}^*\mtx{A}\vct{\xi}'}$ is subgaussian
conditionally on $\vct{\xi}'$. Therefore, by \eqref{eq:subg_mom},
\begin{align*}
\|\inr{\mtx{A}\vct{\xi},\mtx{A}\vct{\xi}'}\|_{L_p} &= \left(\E_{\xi'} \left(\left(\E_{\xi} |\inr{\vct{\xi},\mtx{A}^*\mtx{A}\vct{\xi}'}|^p\right)^{1/p}\right)^p\right)^{1/p}
\lesssim  (\E_{\xi'} (L \sqrt{p})^p\|\mtx{A^*A}\vct{\xi}'\|_2^p)^{1/p}\nonumber\\
& = L \sqrt{p} \left(\E_{\xi'} \sup_{\vct{y} \in S} |\langle \vct{y}, \vct{\xi}' \rangle|^p\right)^{1/p}\;. 
\end{align*}
We will estimate this quantity using Theorem \ref{thm:subgaussian}. We will bound the two terms on the right hand side of \eqref{eq:subgaussian} separately.
For the first term, let $\vct{g}$ denote a standard Gaussian vector and estimate
\begin{align*}
\E  \sup_{\vct{y} \in S} |\langle \vct{y}, \vct{g}\rangle |
&= \E \|\mtx{A}^* \mtx{A} \vct{g}\|_2
\leq \left(\E \|\mtx{A}^* \mtx{A} \vct{g}\|_2^2 \right)^{1/2}
= \|\mtx{A}^*\mtx{A}\|_F \leq \|\mtx{A}\|_F\| \mtx{A}\|\opn.
\end{align*}
For the second term, applying \eqref{eq:subg_mom} again yields
\begin{equation*}
\sup_{\vct{y} \in S} (\E |\langle \vct{y}, \vct{\xi}'\rangle |^p)^{1/p} = \sup_{\vct{z} \in B_2^n} (\E |\langle \mtx{A}^* \mtx{A} \vct{z}, \vct{\xi}' \rangle |^p)^{1/p} \lesssim L \sup_{\vct{z} \in B_2^n} \sqrt{p} \|\mtx{A}^* \mtx{A} \vct{z}\|_2 = L \sqrt{p} \|\mtx{A}\|\opn^2.
\end{equation*}
 Hence the lemma follows by applying Theorem \ref{thm:subgaussian} and taking the supremum over $\mtx{A} \in {\cal A}$.
 \endproof
The second lemma concerns first moments for the special case of a Gaussian vector.
\begin{Lemma}\label{lem:gauss1mom}
 Let  $\vct{g}$ denote a standard Gaussian vector. Then
 \begin{equation*}
 \E N_A(\vct{g}) \lesssim_L \gamma_2({\cal A},\| \cdot \|_{2 \to 2}) +d_F({\cal A})
 \end{equation*}
\end{Lemma}
\proof
Observe that
by Theorem \ref{thm:decouple:Gaussian} and Lemma~\ref{lemma:chaining}, 
\begin{align}\nonumber
\left\|C_{\cal A}(\vct{g})\,\right\|_{L_p} &= \left\|\sup_{\mtx{A} \in {\cal A}} \left| \sum_{\substack{j,k\\j\neq k}}
g_j g_k \inr{\mtx{A}^j,\mtx{A}^k}+ \sum_j (g_j^2-1) \|\mtx{A}^j\|^2_2 \right|\,\right\|_{L_p}\\
&\lesssim   \left\| \sup_{\mtx{A} \in {\cal A}} \left| \sum_{j,k}
g_j g_k' \inr{\vct{A}^j,\vct{A}^k}\right| \,\right\|_{L_p}=   \left\| \sup_{\mtx{A}\in {\cal A}} \left| \inr{ \mtx{A}\vct{g}, \mtx{A}\vct{g}'}\right|\,\right\|_{L_p}\nonumber\\
& \lesssim_L  \gamma_2({\cal A}, \| \cdot \|_{2 \to 2})\left\| N_{\cal A}(\vct{g})\right\|_{L_p} + \sup_{\mtx{A} \in {\cal A}} \left\| |\inr{\mtx{A} \vct{g},\mtx{A} \vct{g}'}|\,\right\|_{L_p}. \label{eq:decchain_gauss}
\end{align}
Combining Lemma~\ref{lem:decbound} with \eqref{eq:decchain_gauss}, it follows that
\begin{equation}\label{est:CA:Gauss1}
\left\| C_{\cal A} (\vct{g})\right\|_{L_p} \lesssim_{L} \gamma_2({\cal A},\| \cdot \|_{2 \to 2})\left\| N_{\cal A}(\vct{g}) \right\|_{L_p} + \sqrt{p}d_F({\cal A}) d\opn({\cal A})+ p d_{2 \to 2}^2({\cal A}).
\end{equation}
Specifying $p=1$ and using that $d_F({\cal A})\geq d\opn({\cal A})$ as well as $\E\|\mtx{A}\vct{g}\|_2^2=\|\mtx{A}\|_F^2$, we conclude that
\begin{equation*}
\E N_{\cal A}^2 (\vct{g}) \leq \E C_{\cal A} (\vct{g})+d_F^2({\cal A})\lesssim_{L} \gamma_2({\cal A},\| \cdot \|_{2 \to 2})\E N_{\cal A}(\vct{g}) +d_F^2({\cal A}).
\end{equation*}
Therefore, 
\begin{equation*}
\E N_A(\vct{g}) \leq(\E N_A^2(\vct{g}) )^{1/2} \lesssim_L \gamma_2({\cal A},\| \cdot \|_{2 \to 2}) +d_F({\cal A}),
\end{equation*}
as desired. \endproof

\proofof{Theorem~\ref{thm:moment-main}}
We apply Theorem~\ref{thm:subgaussian} with the set $S=\{\mtx{A^*}\vct{x}: \vct{x} \in B_2^n, \ \mtx{A} \in {\cal A} \}$. Since $\vct{\xi}$ is $L$-subgaussian we obtain
\begin{align*}
 \|N_{\cal A}(\vct{\xi})\|_{L_p} &= (\E \sup_{\mtx{A} \in {\cal A}, \vct{x} \in B_2^n} |\inr{\mtx{A}\vct{\xi}, \vct{x}}|^p)^{1/p}
= (\E \sup_{\vct{u} \in S} |\inr{\vct{\xi}, \vct{u}}|^p)^{1/p}\\
&\lesssim_L  \E N_{\cal A}(\vct{g}) + \sup_{\vct{u} \in S} (\E|\inr{\vct{\xi},\vct{u}}|^p)^{1/p}
\lesssim_L\E N_{\cal A} (\vct{g})+ \sqrt{p} \sup_{\mtx{A} \in {\cal A}, \vct{x} \in B_2^n} \|\mtx{A^*}\vct{x}\|_2
\\
&\lesssim_L  \E N_{\cal A}(\vct{g}) + \sqrt{p} d_{2 \to 2}({\cal A}),\\
&\lesssim_L \gamma_2({\cal A},\| \cdot \|_{2 \to 2}) + d_F({\cal A})+ \sqrt{p} d_{2 \to 2}({\cal A}), 
\end{align*}
which proves (a).

For (b), observe that as the $\xi_j$ are unit variance, we have $\E\|\mtx{A}\vct{\xi}\|_2^2=\|\mtx{A}\|_F^2=\sum\limits_{j=1}^n\| \vct{A}^j\|^2_2 $ and, consequently,
$C_{\cal A}({\vct{\xi}})$ can be split up into the diagonal and the off-diagonal contributions as follows.
 \begin{equation*}
C_{\cal A}(\vct{\xi})= \sup_{\mtx{A} \in {\cal A}}\left| \sum_{\substack{j,k=1\\j\neq k}}^n \xi_j \overline{\xi_k}
 \langle \vct{A}^j, \vct{A}^k\rangle+\sum_{j=1}^n (|\xi_j|^2 -1)
 \| \vct{A}^j\|^2 \right|
 \leq B_{\cal A}(\vct{\xi}) + D_{\cal A}(\vct{\xi}).
\end{equation*}
Hence it suffices to estimate the moments of  $B_{\cal A}(\vct{\xi})$ and $D_{\cal A}(\vct{\xi})$; one concludes using the triangle inequality.

For the off-diagonal term, we use Theorem \ref{thm:decouple} and Lemma~\ref{lemma:chaining} to bound 
\begin{align*}
\left\| B_{\cal A}(\vct{\xi})\right\|_{L_p} & \leq 4 \left\|\sup_{\mtx{A} \in {\cal A}} \left| \sum_{j,k}
\xi_j \overline{\xi_k'} \inr{\vct{A}^j,\vct{A}^k}\right|\right\|_{L_p} =  4 \left\| \sup_{\mtx{A}\in {\cal A}} \left| \inr{ \mtx{A}\vct{\xi}, \mtx{A}\vct{\xi}'}\right|\right\|_{L_p} \\&  \lesssim_L \gamma_2({\cal A}, \| \cdot \|_{2 \to 2})\|N_{\cal A}(\vct{\xi})\|_{L_p} + \sup_{\mtx{A} \in {\cal A}} \| \inr{\mtx{A} \vct{\xi},\mtx{A} \vct{\xi}'}\|_{L_p}.
\end{align*} 

Combining this estimate with Lemma~\ref{lem:decbound} and part (a), we obtain that 
\begin{align}
\| B_{\cal A}(\vct{\xi}) \|_{L_p}  \lesssim_L& \gamma_2({\cal A},\|\cdot\|_{2 \to 2})\left(\gamma_2({\cal A},\|\cdot\|_{2 \to 2})  + d_F({\cal A})
 + \sqrt{p} d_{2 \to 2}({\cal A})\right)\notag\\
& + \sqrt{p} d_F({\cal A}) d_{2 \to 2}({\cal A})
+ p d_{2 \to 2}^2({\cal A}). \label{eq:BA}
\end{align}

For the diagonal term, observe that by a standard  symmetrization argument (see, e.g.,\ \cite[Lemma 6.3]{leta91}),
\[
\left\|D_{\cal A}(\vct{\xi}) \right\|_{L_p}= \left\|\sup_{\mtx{A} \in {\cal A}} \left|\sum_{j=1} (|\xi_j|^2 - \E |\xi_j|^2) \|\mtx{A}^j\|_2^2\right|\,\right\|_{L_p}
\leq 2 \left\|\sup_{\mtx{A} \in {\cal A}} \left| \sum_j \epsilon_j |\xi_j|^2 \|\mtx{A}^j\|_2^2 \right|\,\right\|_{L_p}
\]
where $\vct{\epsilon} = (\epsilon_1,\hdots,\epsilon_n)$ is a Rademacher vector independent of $\vct{\xi}$. Furthermore, let $\vct{g} = (g_1,\hdots,g_n)$
be a sequence of independent standard normal variables. Then,
as $\xi_j$ is $L$-subgaussian, there is an absolute constant $c$ for which $\P(|\xi_j|^2 \geq tL^2) \leq c \P(g_j^2 \geq t)$ for every $t > 0$.
 Moreover, $\epsilon_j|\xi_j|^2$ and $\epsilon_j g_j^2$ are symmetric, so by the contraction principle (see \cite[Lemma 4.6]{leta91}), a rescaling argument, 
and de-symmetrization \cite[Lemma 6.3]{leta91},
\begin{align*}
\left\| D_{\cal A}(\vct{\xi})\,\right\|_{L_p} & \lesssim_L  \left\| \sup_{\mtx{A} \in {\cal A}} | \sum_j \epsilon_j g_j^2\|\vct{A}^j\|_2^2|\,\right\|_{L_p}\\& \leq
2 \left\| \sup_{\mtx{A} \in {\cal A}} | \sum_j (|g_j|^2 - 1) \|\vct{A}^j\|_2^2|\,\right\|_{L_p} + \left\|\sup_{\mtx{A} \in {\cal A}} | \sum_j \epsilon_j \|\vct{A}^j\|_2^2|\,\right\|_{L_p}
\\
& = 2 \left\| D_{\cal A}(\vct{g})\right\|_{L_p} + \left\| \sup_{\mtx{A} \in {\cal A}} | \sum_j \epsilon_j \|\vct{A}^j\|_2^2|\,\right\|_{L_p}.
\end{align*}
Now observe that $D_{\cal A}(\vct{g}) \leq C_{{\cal A}}(\vct{g}) + B_{{\cal A}}(\vct{g})$, and thus, by  \eqref{est:CA:Gauss1} and \eqref{eq:BA},
\begin{align}
 \left\| D_{\cal A}(\vct{g})\,\right\|_{L_p}\leq & \left\| C_{\cal A}(\vct{g})\,\right\|_{L_p}+ \left\| B_{\cal A}(\vct{g})\,\right\|_{L_p}\notag\\
 \lesssim& \gamma_2({\cal A},\|\cdot\|_{2 \to 2})\left(\gamma_2({\cal A},\|\cdot\|_{2 \to 2}) + d_F({\cal A}) \right) \notag\\&+ \sqrt{p} d_{2 \to 2}({\cal A}) (d_{F}({\cal A}) +\gamma_2({\cal A},\|\cdot\|_{2 \to 2}))+ p d^2\opn({\cal{A}}).\notag
\end{align}

Finally, note that ${\mtx A} \to \sum_j \epsilon_j \|\vct{A}^j\|_2^2$ is a subgaussian process relative to the metric
\begin{align}
d(\mtx{A},\mtx{B}) & =  \left(\sum_{j=1}^n (\|\vct{A}^j\|_2^2-\|\vct{B}^j\|_2^2)^2 \right)^{1/2} \nonumber
\\
\leq & \left(\sum_{j=1}^n \|\vct{A}^j-\vct{B}^j\|_2^2 \cdot (\|\vct{A}^j\|_2+\|\vct{B}^j\|_2)^2 \right)^{1/2} \leq 2d_F({\cal A}) \|{\mtx A}-{\mtx B}\|\opn. \label{eq:metric}
\end{align}
Therefore, by Theorem~\ref{thm:subgaussian} and a standard chaining argument, 
\[
\left\| \sup_{{\mtx A} \in {\cal A}} | \sum_j \epsilon_j \|\vct{A}^j\|_2^2|\,\right\|_{L_p} \lesssim d_F({\cal A}) \gamma_2({\cal A}, \|\cdot\|\opn) + \sqrt{p} \,d_F({\cal A})d\opn({\cal A}).
\]
Here the second term corresponds to the second term in \eqref{eq:subgaussian}, which is bounded using \eqref{eq:subg_mom} and \eqref{eq:metric}.
This shows that
 \begin{align*}
\left\| D_{\cal A}(\vct{\xi})\right\|_{L_p} \lesssim_L & \gamma_2({\cal A}, \| \cdot \|_{2 \to 2})\left(\gamma_2({\cal A}, \| \cdot \|_{2 \to 2})+d_F({\cal A}) \right).
\\
&  + \sqrt{p} d_{2 \to 2}({\cal A})\left( \gamma_2({\cal A}, \| \cdot \|_{2 \to 2}) +d_{F}({\cal A})\right)+ pd^2_{2 \to 2}({\cal A}),
\end{align*}
which, together with \eqref{eq:BA}, proves (b).
\endproof

\begin{Remark}\label{remRademacher}
\begin{itemize}
\item[(a)] Observe that Theorem \ref{thm:main-tail} can be deduced
from Theorem \ref{thm:moment-main} using Proposition~\ref{prop:mom2tail}.
\item[(b)] In the Rademacher case, one has $D_{\cal{A}}\equiv 0$, so the contraction principle and the more sophisticated decoupling inequality for Gaussian random variables are not needed in the proof.
\item[(c)] Note that the assumption that $\vct{\xi}$ has independent coordinates has only been used in the decoupling steps of
the proof of Theorem \ref{thm:moment-main}.
\end{itemize}
\end{Remark}

\section{The Restricted Isometry Property of Partial Random Circulant Matrices\label{sec:RIPcirc}}

In this section we study the restricted isometry constants of a partial random circulant matrix $\mtx{\Phi} \in \R^{m \times n}$ generated by a random vector $\vct{\xi} = (\xi_i)_{i=1}^n$, where the $\xi_i$'s are independent mean-zero, $L$-subgaussian random variables of variance one. Arguably the most important case is
when 
$\vct{\xi} = \vct{\epsilon}$ is a Rademacher
vector, as introduced in Section \ref{Sec:PRCM}.

Throughout this section and following the notation of the introduction, let $\mtx{V}_{\vct{x}} \vct{z} = \frac{1}{\sqrt{m}} \mtx{P}_\Omega(\vct{x} * \vct{z})$,
where the projection operator $\mtx{P}_\Omega: \C^n \to \C^n$
is given by $\mtx{P}_\Omega=\mtx{R}^*_\Omega\mtx{R}_\Omega$, that is, $(\mtx{P}_\Omega \vct{x})_\ell = x_\ell$ for $\ell \in \Omega$ and $(\mtx{P}_\Omega \vct{x})_\ell = 0$ for $\ell \notin \Omega$.

Setting $D_{s,n} = \{\vct{x} \in \C^n: \|\vct{x}\|_2 \leq 1, \|\vct{x}\|_0 \leq s\}$, the restricted isometry constant
of $\mtx{\Phi}$ is
\begin{align}
\delta_s &= \sup_{\vct{x} \in D_{s,n}}\left|\frac{1}{m}\|\mtx{R}_\Omega( \vct{\xi}*\vct{x})\|^2_2-\|\vct{x}\|_2^2\right|=\sup_{\vct{x} \in D_{s,n}}\left|\frac{1}{m}\|\mtx{P}_\Omega( \vct{x}*\vct{\xi})\|^2_2-\|\vct{x}\|_2^2\right|\notag\\
&=\sup_{\vct{x} \in D_{s,n}} | \|\mtx{V}_{\vct{x}} \vct{\xi}\|_2^2 - \|\vct{x}\|_2^2|.\notag
\end{align}
Since $|\Omega|=m$, it follows that
\[
\E \|\mtx{V}_{\vct{x}} \vct{\xi}\|_2^2 = \frac{1}{m} \sum_{\ell \in \Omega} \E \sum_{k,j=1}^n \xi_j \overline{\xi_k} x_{\ell \ominus j} \overline{x_{\ell \ominus k}}
= \frac{1}{m} \sum_{\ell \in \Omega} \sum_{k=1}^n |x_{\ell \ominus k}|^2 = \|\vct{x}\|_2^2
\]
and hence
\[
\delta_s = \sup_{\vct{x} \in D_{s,n}} \left| \|\mtx{V}_{\vct{x}} \vct{\xi}\|_2^2 - \E \|\mtx{V}_{\vct{x}} \vct{\xi}\|_2^2\right|,
\]
which shows that $\delta_s$ is the process $C_{\cal{A}}$ studied in the previous section for ${\cal A} = \{ \mtx{V}_{\vct{x}}: \vct{x} \in D_{s,n}\}$; thus the tail decay can be 
analyzed using Theorem~\ref{thm:main-tail}.

\begin{Theorem}\label{thm:ERIP}
Let $\vct{\xi} = (\xi_j)_{j=1}^n$ be a random vector with independent, mean-zero, variance one, $L$-subgaussian entries. 
If, for $s\leq n$ and $\eta,\delta \in(0,1)$,
\begin{equation}\label{m:bound:lower:circ}
m \geq c \delta^{-2} s \max \{ (\log s)^2(\log n)^2, \log( \eta^{-1}) \},
\end{equation}
then with probability at least $1-\eta$, the restricted isometry constant of the partial random circulant matrix
$\mtx{\Phi} \in \R^{m \times n}$
generated by $\vct{\xi}$ satisfies $\delta_s \leq \delta$. The constant $c > 0$ depends only on $L$.
\end{Theorem}

The proof of Theorem \ref{thm:ERIP} requires a Fourier domain description of $\mtx{\Phi}$.
Let $\mtx{F}$ be the unnormalized
Fourier transform with elements $F_{jk} = e^{2\pi i jk / n}$. By the convolution
theorem, for every $1 \leq j \leq n$, $\mtx{F}(\vct{x}*\vct{y})_j = ( \mtx{F} \vct{x})_j \cdot (\mtx{F} \vct{y})_j$. Therefore,
\[
\mtx{V}_{\vct{x}} \vct{\xi} = \frac{1}{\sqrt{m}} \mtx{P}_\Omega \mtx{F}^{-1} \widehat{\mtx{X}} \mtx{F} \vct{\xi},
\]
where $\widehat{\mtx{X}} = {\rm diag}(\mtx{F} \vct{x})$ is the diagonal matrix, whose diagonal is the Fourier transform $\mtx{F} \vct{x}$.
 In short,
\[
\mtx{V}_{\vct{x}} = \frac{1}{\sqrt{m}} \widehat{\mtx{P}}_\Omega \widehat{\mtx{X}} \mtx{F},
\]
where $\widehat{\mtx{P}}_\Omega = \mtx{P}_\Omega \mtx{F}^{-1}$.

\vskip6pt
\noindent{\bf Proof of Theorem \ref{thm:ERIP}}.
In light of Theorem \ref{thm:moment-main} and Theorem \ref{thm:main-tail}, it suffices to control the parameters $d_{2 \to 2}({\cal A})$, $d_{F}({\cal A})$, and $\gamma_2({\cal A},\| \cdot \|_{2 \to 2})$
for the set
${\cal A} = \{\mtx{V}_{\vct{x}}: \vct{x} \in D_{s,n}\}$.

Since the matrices $\mtx{V}_{\vct{x}}$ consist of shifted copies of $\vct{x}$ in all of their $m$ nonzero rows, the $\ell_2$-norm of each nonzero row is $m^{-1/2} \|\vct{x}\|_2$; thus
$\|\mtx{V}_{\vct{x}}\|_F = \|\vct{x}\|_2 \leq 1$ for all $\vct{x} \in D_{s,n}$ and
\[
d_F({\cal A}) = 1.
\]
Also, observe that for every $\vct{x} \in D_{s,n}$ with the associated diagonal
matrix $\widehat{\mtx{X}}$,
\begin{align}\nonumber
\|\mtx{V}_{\vct{x}}\|_{2 \to 2}  & =  \frac{1}{\sqrt{m}} \|\widehat{\mtx{P}}_{\Omega} \widehat{\mtx{X}} \mtx{F}\|_{2 \to 2}
\leq \sqrt{\frac{n}{m}}\|\mtx{P}_{\Omega} \mtx{F}^{-1}\|_{2 \to 2} \| \widehat{\mtx{X}} \|_{2 \to 2}
\leq \frac{1}{\sqrt{m}} \|\widehat{\mtx{X}} \|_{2 \to 2} \\
&= \frac{1}{\sqrt{m}} \|\mtx{F}\vct{x}\|_\infty.\label{eq:opndiam}
\end{align}
Setting $\|\vct{x}\|_{\widehat{\infty}}:=\|\mtx{F}\vct{x}\|_\infty$ it is evident that  $\|\mtx{F} \vct{x}\|_\infty \leq \|\vct{x}\|_1 \leq \sqrt{s} \|\vct{x}\|_2 \leq \sqrt{s}$ for every $x \in D_{s,n}$, and thus
\[
d_{2 \to 2}({\cal A}) \leq \sqrt{s/m}.
\]
Next, to estimate the $\gamma_2$ functional, recall from \eqref{eq:Dudley} that
\[
\gamma_2({\cal A},\| \cdot \|_{2 \to 2}) \lesssim \int_0^{d_{2 \to 2}({\cal A})} \log^{1/2}N({\cal A},\| \cdot \|_{2 \to 2},u) du.
\]
By \eqref{eq:opndiam}, 
 \begin{equation*}
  \|\mtx{V}_{\vct{x}}-\mtx{V}_{\vct{y}}\|_{2 \to 2}=\|\mtx{V}_{\vct{x}-\vct{y}}\|_{2 \to 2} \leq m^{-1/2}\|{\vct{x}}-{\vct{y}}\|_{\widehat{\infty}},
 \end{equation*}
and hence for every $u>0$, $N({\cal A},\| \cdot \|_{2 \to 2},u) \leq N(D_{s,n},m^{-1/2}\| \cdot \|_{\widehat{\infty}},u)$.

To estimate this covering number, let us recall the following simple modification of the Maurey Lemma, essentially due to Carl \cite{Ca85} (see also \cite{ruve08} or \cite[Lemma 8.3]{ra10} for refined estimates specific to $D_{s,n}$). For convenience, a proof is provided in
the appendix. Below, for a set ${\cal U}$ in a vector space,
$\conv({\cal U})$ denotes its convex hull.
\begin{Lemma} \label{lemma:maurey}
There exists an absolute constant $c$ for which the following holds. Let $X$ be a normed space, consider a finite set ${\cal U} \subset X$ of cardinality $N$, and assume that for
every $L \in \N$ and $(\vct{u}_1,\hdots,\vct{u}_L) \in {\cal U}^L$, $\E_\eps \|\sum_{j=1}^L \eps_j \vct{u}_j \|_X \leq A\sqrt{L}$,
where $(\eps_j)_{j=1}^L$ denotes a Rademacher vector. Then for every $u>0$,
$$
\log N({\conv}({\cal U}), \| \cdot \|_X, u) \leq c (A/u)^2 \log N.
$$
\end{Lemma}

We set $\Delta=\sqrt{s/m}$ and apply the lemma for the set \[{\cal U}=\{\pm\sqrt{2} e_1, \dots, \pm \sqrt{2}e_n, \pm\sqrt{2}i e_1, \dots, \pm \sqrt{2}i e_n\},\] where the $e_i$ are the standard basis vectors. Noting that $B_1^n\subset {\conv}({\cal U})$ and that we may choose $A = c \sqrt{\log(n)}$, we obtain in this case
\begin{align*}
\log N(D_{s,n},m^{-1/2}\| \cdot \|_{\widehat{\infty}},u) \leq & \log N(s^{1/2}B_1^n,m^{-1/2}\| \cdot \|_{\widehat{\infty}},u)
\\
\lesssim & \left(\frac{\Delta}{u}\right)^2 \log^2 (n).
\end{align*}
Since $D_{s,n}$ is the union of $s$-dimensional Euclidean balls, a standard volumetric argument
(see, e.g.,\ \cite{ruve08} or \cite[Chapter 8.4]{ra10}) yields 
\begin{equation*}
\log N(D_{s,n},m^{-1/2}\| \cdot \|_{\widehat{\infty}},u) \lesssim s\log(en/su)
\end{equation*}
(which is stronger than the bound above for $u \lesssim 1/\sqrt{m}$).

Combining the two covering number estimates, a straightforward computation of the entropy integral,
(see also \cite{ruve08} or \cite[eq.\ (8.15)]{ra10}), reveals that
\begin{equation*}
\gamma_2({\cal A},\| \cdot \|_{2 \to 2}) \lesssim \sqrt{\frac{s}{m}} (\log s) (\log n),
\end{equation*}
which implies that $\gamma_2({\cal A},\| \cdot \|_{2 \to 2}) \lesssim \delta$ for the given choice of $m$.

Now, by choosing the constant $c$ in \eqref{m:bound:lower:circ}
appropriately (depending only on $L$), one obtains
\begin{equation*}
E\leq \frac{\delta}{2c_1},
\end{equation*}
where $E$ and $c_1$ are chosen as in Theorem \ref{thm:main-tail}. Then  Theorem \ref{thm:main-tail} yields
\begin{equation*}
\P(\delta_s \geq \delta) \leq \P \left(\delta_s \geq c_1 E  + \delta/2 \right) \leq \exp(-c_2(m/s)\delta^2)
\leq \eta,
\end{equation*}
which, after possibly increasing the value of $c$ enough to compensate $c_2$, completes the proof.
\endproof

\section{Time-Frequency Structured Random Matrices\label{sec:RIPgabor}}

In this section, we will treat the restricted isometry property of random Gabor synthesis matrices,
as described in Section \ref{Sec:Intro:Gabor}.

\begin{Theorem} \label{thm:gabor}
Let $\vct{\xi} = (\xi_j)_{j=1}^m$ be a random vector with independent mean-zero,
variance one, $L$-subgaussian entries. Let $\Psi_{\vct{h}} \in \C^{m \times m^2}$
be the random Gabor synthesis matrix generated by $\vct{h} = \frac{1}{\sqrt{m}}\vct{\xi}$.
If, for $s \in \N$ and $\delta,\eta \in (0,1)$,
\[
m \geq c \delta^{-2} s \max\{(\log^2 s) (\log^2 m), \log(\eta^{-1})\}
\]
then with probability at least $1-\eta$
the restricted isometry constant of $\Psi_{\vct{h}}$ satisfies
that $\delta_s \leq \delta$. The constant $c>0$ depends only on $L$.
\end{Theorem}
Before presenting the proof we will need several observations. First, note that for $\vct{x} \in \C^m$, $\mtx{\Psi}_{\vct{h}} \vct{x} = \mtx{V}_{\vct{x}} \vct{\xi}$, where the $m \times m$ matrix $\mtx{V}_{\vct{x}}$ is given by
$$
\mtx{V}_{\vct{x}} =
\frac{1}{\sqrt{m}} \sum_{\lambda \in \Z_m^2} x_{\lambda} \mtx{\pi}(\lambda).
$$
It is straightforward to check that
$\{\frac{1}{\sqrt{m}}\mtx{\pi}(\lambda) : \lambda \in \Z_m^2\}$
is an orthonormal system
in the space of complex $m \times m$ matrices endowed with the Frobenius norm.
Therefore,
$$
\E \|\mtx{V}_{\vct{x}} \vct{\xi}\|_2^2 = \|\mtx{V}_{\vct{x}}\|_F^2 = \|m^{-1/2} \sum_{\lambda \in \Z_m^2} x_\lambda \mtx{\pi}(\lambda)\|_F^2  = \|\vct{x}\|_2^2.
$$
Hence, if $D_{s,n} = \{\vct{x} \in \C^{m^2}: \|\vct{x}\|_2 \leq 1, \|\vct{x}\|_0 \leq s\}$ and ${\cal A} = \{\mtx{V}_{\vct{x}}: \vct{x} \in D_{s,n}\}$, the restricted isometry constant is
$$
\delta_{s} = \sup_{\vct{x} \in D_{s,n}} \left| \|\mtx{\Psi}_{\vct{h}} \vct{x}\|_2^2 - \|\vct{x}\|_2^2 \right|
= \sup_{\mtx{V}_{\vct{x}} \in {\cal A}} \left| \|\mtx{V}_{\vct{x}} \vct{\xi}\|_2^2 - \E \|\mtx{V}_{\vct{x}} \vct{\xi}\|_2^2 \right|.
$$
Thus Theorem~\ref{thm:main-tail} applies again and it suffices to estimate the associated Dudley integral. 
Note that, as $\mtx{\pi}(\lambda)$ is unitary, one has for $\vct{x}\in D_{s,n}$
\begin{equation}
\|\mtx{V}_{\vct{x}}\|_{2 \to 2} 
\leq \frac{1}{\sqrt{m}} \sum_{\lambda \in \Z_m^2} |x_{\lambda}|\, \|\mtx{\pi}(\lambda)\|_{2 \to 2} 
\leq \|\vct{x}\|_1/\sqrt{m} \leq \sqrt{s/m}\, \|\vct{x}\|_2, \label{eq:tfdiam}
\end{equation}
so the upper limit in the Dudley integral is $d\opn({\cal A})\leq\sqrt{\frac{s}{m}}$.
\begin{Lemma} \label{lemma:entropy-Gabor}
There exists an absolute constant $c$ such that for every $0<u\leq \sqrt{\frac{s}{m}}$,
$$
\log N({\cal A}, \| \cdot \|_{2 \to 2}, u) \leq cs\left( \log (em^2/s) + \log(3 \sqrt{s/m}/u)\right),
$$
and
$$
\log N({\cal A}, \| \cdot \|_{2 \to 2}, u) \leq c \frac{s \log^2 m}{mu^2}.
$$
\end{Lemma}

\proof
Define the norm $\|\cdot\|$ on $\R^n$ by $\|\vct{x}\|=\|\mtx{V}_{\vct{x}}\|_{2 \to 2}$, fix $S \subset \Z_m^2$ of cardinality $s$ and put $B_S = \{\vct{x} \in \C^{m^2}: \|\vct{x}\|_2\leq 1, {\rm supp}(\vct{x}) \subset S\}$. Then, by \eqref{eq:tfdiam} and a volumetric estimate,
$$
N(B_S, \|\cdot\|,u) \leq N(B_S, \sqrt{s/m} \|\cdot\|_2, u) \leq \left(1+2\frac{\sqrt{s/m}}{u}\right)^{2s}\leq \left(3\frac{\sqrt{s/m}}{u}\right)^{2s},
$$
where the last step uses that $u\leq\sqrt{s/m}$. Since there are at most $\binom{m^2}{s} \leq (em^2/s)^s$
such subsets $S$ of $\Z_m^2$, the first part of the claim follows.

To prove the second part, note that
$D_{s,n} \subset \sqrt{2s}\left({\rm conv}(\vct{e}_\lambda, i\vct{e}_\lambda,-\vct{e}_\lambda,-i\vct{e}_\lambda)_{\lambda \in \Z_m^2}\right)$. 
Consider $(\vct{u}_j)_{j=1}^L$ selected from the extreme points (with possible repetitions). Then, by the non-commutative Khintchine inequality, due to Lust-Piquard, \cite{lu86-1,lupi91,ru99},
$$
\E_{\vct{\eps}} \| \sum_{j=1}^L \eps_j \mtx{V}_{\vct{u}_j} \|_{2 \to 2} \lesssim
\sqrt{\log m}\, \max\left\{\|\sum_{j=1}^L  \mtx{V}_{\vct{u}_j} \mtx{V}_{\vct{u}_j}^*\|_{2 \to 2},
\|\sum_{j=1}^L  \mtx{V}_{\vct{u}_j}^* \mtx{V}_{\vct{u}_j}\|_{2 \to 2}\right\}^{1/2}.
$$
Recall that for every such $\vct{u}_j$, $\mtx{V}_{\vct{u}_j}=\alpha \mtx{\pi}(\lambda)$ with $|\alpha|=\sqrt{2s/m}$. Therefore, $\mtx{V}^*_{\vct{u}_j}\mtx{V}_{\vct{u}_j}=\mtx{V}_{\vct{u}_j}\mtx{V}_{\vct{u}_j}^*=(2s/m)\mtx{I}$ and thus $$
\E_{\vct{\eps}} \| \sum_{j=1}^L \eps_j \mtx{V}_{\vct{u}_j} \|_{2 \to 2} \lesssim \sqrt{\log m}
 \sqrt{s/m} \sqrt{L}.
$$
Applying Lemma \ref{lemma:maurey} for $A \sim \sqrt{s/m}\sqrt{\log m}$,
it follows that
$$
\log N(D_{s,n}, \| \cdot \|, u) \lesssim (A/u)^2 \log(m^2)
\lesssim \frac{s \log^2 m}{m u^2}.
$$
\endproof
\noindent {\bf Proof of Theorem \ref{thm:gabor}.}
The proof follows an identical path to that of Theorem \ref{thm:ERIP}. First, as was noted above, $d_F({\cal A}) \leq 1$ and $d_{2 \to 2}({\cal A}) \leq \sqrt{s/m}$. Also, using
the bound \eqref{eq:Dudley} by the Dudley type integral 
and by a direct application of Lemma~\ref{lemma:entropy-Gabor},
$$
\gamma_2({\cal A},\| \cdot \|_{2 \to 2}) \lesssim \int_0^{d_{2 \to 2}({\cal A})} \sqrt{\log N({\cal A},\| \cdot \|, u)} du \lesssim \sqrt{s/m} (\log s)(\log m).
$$
Here we used the first bound of Lemma~\ref{lemma:entropy-Gabor} for $u\lesssim m^{-1/2}$, the second bound for $u\gtrsim m^{-1/2}$.
The claim is now a direct application 
of Theorem \ref{thm:main-tail}. \endproof

\begin{Remark} The only properties of the system $\{\mtx{\pi}(\lambda) : \lambda \in \Z_m^2\}$
that have been used in the proof are the facts that all $\mtx{\pi}(\lambda)$ are unitary and that $\{m^{-1/2} \mtx{\pi}(\lambda) : \lambda \in \Z_m^2\}$
is an orthonormal system with respect to the Frobenius inner product. Therefore, Theorem~\ref{thm:gabor} also holds true for general systems of operators with these two properties.
\end{Remark}
\appendix
\section{Appendix}

\subsection{Proof of Theorem \ref{thm:subgaussian}}
\label{AppThm:moments:subgauss}

Without loss of generality assume that $T$ is finite.
Fix an optimal admissible sequence $(T_r)$ of $T$. For $t \in T$, let $\pi_r(\vct{t}) \in T_r$ be an element in $T_r$ with the smallest $\ell_2$-distance to $\vct{t}$, and choose $\ell$ for which $2^{\ell-1} \leq 2 p \leq 2^{\ell}$. Since one may assume that $\pi_r(\vct{t}) = \vct{t}$ for a sufficiently large $r$, one has
\begin{equation}\label{split:p}
\sup_{\vct{t} \in T} |\langle \vct{t}, Y\rangle| \leq \sup_{\vct{t} \in T} |\langle \pi_{\ell}(\vct{t}), Y \rangle| + \sup_{\vct{t} \in T} \sum_{r = \ell}^\infty |\langle \pi_{r+1}(\vct{t}) - \pi_{r}(\vct{t}), Y\rangle|.
\end{equation}
The $p$-th moment of the first term satisfies
\begin{align}
\left( \E \sup_{\vct{t} \in T} |\langle \pi_{\ell}(\vct{t}), Y \rangle|^p\right)^{1/p} & \leq \left(\E \sum_{\vct{t} \in T_{\ell}} |\langle \vct{t}, Y \rangle|^p \right)^{1/p}
\leq (|T_{\ell}|)^{1/p} \sup_{\vct{t} \in T_{\ell}} \left( \E |\langle \vct{t}, Y \rangle|^p\right)^{1/p}\notag\\
& \leq (2^{2^{\ell}})^{1/p} \sup_{\vct{t} \in T} \left( \E |\langle \vct{t}, Y \rangle|^p \right)^{1/p}
\leq 16 \sup_{\vct{t} \in T} \left( \E |\langle \vct{t}, Y \rangle|^p \right)^{1/p}, \notag
\end{align}
where the last inequality follows from the choice of $p$.

Since $\vct{\xi}$ is an $L$-subgaussian vector, one obtains for the second term in \eqref{split:p}
\begin{align}
&\P\left(\sup_{\vct{t} \in T} \sum_{r = \ell}^\infty |\langle \pi_{r+1}(\vct{t}) - \pi_r(\vct{t}), Y \rangle| \geq u L \sum_{r=\ell}^\infty 2^{r/2}\| (\langle \pi_{r+1}(\vct{t}) - \pi_r(\vct{t}), \vct{x}_j\rangle)_{j=1}^n \|_2 \right)\notag\\
& \leq \sum_{r=\ell}^\infty \sum_{\vct{t} \in T_{r+1}} \sum_{\vct{t}' \in T_{r}} \P\left( |\sum_{j=1}^m \xi_j \langle \vct{t}-\vct{t}', \vct{x}_j \rangle| \geq  u L 2^{r/2} \| \langle\vct{t}-\vct{t}', \vct{x}_j\rangle_{j=1}^n \|_2 \right) \notag\\
& \leq \sum_{r=\ell}^\infty 2^{2^{r+1}} \cdot 2^{2^r} \exp(-2^r u^2/2)
\leq 2 \exp( - 2^{\ell} u^2/4) \leq 2 \exp(-p u^2/2), \label{chain:estimate}
\end{align}
when $u \geq c$ for an appropriate choice of $c$ (independent of $\ell$). Therefore, by integration, 
\begin{align}
\left(\E \sup_{t \in T} \sum_{r = \ell}^\infty |\langle \pi_{r+1}(\vct{t}) - \pi_r(\vct{t}), Y \rangle|^p\right)^{1/p} & \lesssim_L \sum_{r = \ell}^\infty  2^{r/2}\| (\langle \pi_{r+1}(\vct{t}) - \pi_r(\vct{t}), \vct{x}_j\rangle)_{j=1}^m \|_2\notag\\
& \lesssim_L \gamma_2(T',\|\cdot\|_2) \notag
\end{align}
where $T' = \{ (\langle \vct{t},\vct{x}_j\rangle)_{j =1}^n : t \in T\}$. By the majorizing measures theorem, 
\[
\gamma_2(T',\|\cdot\|_2) \lesssim \E \sup_{\vct{z} \in T'} | \sum_{j=1}^m z_j g_j  | = \E \sup_{t \in T} | \sum_{j=1}^m g_j \langle \vct{t}, \vct{x}_j\rangle |
= \E \sup_{t \in T} |\langle t, G \rangle|,
\]
which yields the claim.

\subsection{Proof of Lemma \ref{lemma:maurey}}

If $\vct{x} \in \conv({\cal U})$ then $\vct{x} = \sum_{j=1}^N \theta_j \vct{u}_j$
with $\theta_j \geq 0$, $\sum_{j=1}^N \theta_j = 1$. Let $\vct{Z} \in X$ be a random vector which takes the value $\vct{u}_j$ with probability $\theta_j$ for $j=1,\hdots,N$ and thus satisfies $\E \vct{Z} = \vct{x}$.
Let $L$ be a number to be determined later, set $\vct{Z}_1,\hdots,\vct{Z}_L$
be independent copies of $\vct{Z}$, and put
\begin{equation*}
\vct{Y} = \frac{1}{L} \sum_{j=1}^L \vct{Z}_\ell.
\end{equation*}
If $(\epsilon_j)_{j=1}^m$ is a Rademacher sequence independent of $(\vct{Z}_j)$ then
by a standard symmetrization argument (see, e.g.,\ \cite{leta91}) and because $(\vct{Z}_j)_{j=1}^L$ ranges over ${\cal U}^L$
\begin{equation*}
\E \|\vct{x} - \vct{Y}\|_X = \frac{1}{L} \E \|\sum_{j=1}^L (\vct{x} - \vct{Z}_j)\|_X
\leq \frac{2}{L} \E \| \sum_{j=1}^L \epsilon_j \vct{Z}_j \|_X \leq 2 A/\sqrt{L}.
\end{equation*}
Thus, for $L \sim (A/u)^2$
there exists a realization $\vct{y}$ of $Y$ of the form $\vct{y}=\frac{1}{L} \sum_{j=1}^L \vct{z}_\ell$, for some $\vct{z}_\ell\in {\cal U},$ of $\vct{Y}$  for which
\[
\| \vct{x} - \vct{y}\|_X \leq u.
\]
As this argument applies for any $\vct{x} \in \conv({\cal U})$, any such
$\vct{x}$ can be approximated by some $\vct{y}$ of this form. Since
 $\vct{y}$ can assume at most $N^L$ different values, this yields
\[
\log N(\conv({\cal U}), \|\cdot\|_X,u) \leq L \log N \leq c (A/u)^2 \log N,
\]
as claimed.

\subsection{Restricted Isometry Property of Subgaussian Random Matrices}
\label{RIP:subgauss}

We finally show that our bound for chaos processes yields an alternative proof 
of the well-known fact that subgaussian matrices (including Bernoulli and Gaussian matrices) satisfy the restricted isometry property.

An $m \times n$ subgaussian matrix $\mtx{\Phi}$ takes the form
\[
\Phi_{jk} = \frac{1}{\sqrt{m}} \xi_{jk}, \quad j =1,\hdots,m,\; k =1,\hdots,n,
\]
where the $\xi_{jk}$ are independent, mean-zero, variance one, $L$-subgaussian random variables.

\begin{Theorem} Let $\delta, \varepsilon \in (0,1)$. 
A random draw of an $m \times n$ subgaussian random matrix satisfies $\delta_s \leq \delta$ with probability
at least $1-\varepsilon$ provided
\begin{equation}\label{cond:Bernoulli:m}
m \geq C \delta^{-2} \max\{ s \log(en/s), \log(\varepsilon^{-1})\}.
\end{equation}
The constant $C>0$ depends only on the subgaussian parameter $L$.
\end{Theorem}
\begin{proof} We write
\[
\mtx{\Phi} \vct{x} = \mtx{V}_{\vct{x}} \vct{\xi},
\]
where $\vct{\xi}$ is an $L$-subgaussian random vector of length $mn$ and $\mtx{V}_{\vct{x}}$ is the $m \times nm$ block-diagonal
matrix 
\[
\mtx{V}_{\vct{x}} = \frac{1}{\sqrt{m}} \left( \begin{matrix} \vct{x}^T & 0 & \cdots & 0\\
0 & \vct{x}^T & \cdots & 0\\
\vdots & \vdots & \vdots & \vdots\\
0 & \cdots & 0 & \vct{x}^T
\end{matrix} \right).
\]
Clearly, $\|\mtx{V}_{\vct{x}}\|_F = \|\vct{x}\|_2$ so that ${\cal A} := \{\mtx{V}_{\vct{x}}: \vct{x} \in D_{s,n}\}$ has radius
\[
d_F({\cal A}) = 1.
\]
Moreover, since the operator norm of a block-diagonal matrix is the maximum of the operator norms of the diagonal blocks
and the operator norm of a vector is its $\ell_2$-norm we have
\[
\|\mtx{V}_{\vct{x}}\|\opn = \frac{1}{\sqrt{m}}\|\vct{x}\|_2\;.
\]
Hence, $d_{2\to 2}({\cal A}) = \frac{1}{\sqrt{m}}$.
The bound of the $\gamma_2$-functional via the Dudley type integral yields
\[
\gamma_2({\cal A},\|\cdot\|\opn) \leq C \int_0^{1/\sqrt{m}} \sqrt{\log N(D_{s,n},\|\cdot\|_2/ \sqrt{m},u)} du 
= C \frac{1}{\sqrt{m}} \int_0^1 \sqrt{\log N(D_{s,n},\|\cdot\|_2,u)} du.
\]
The volumetric argument yields
\[
N(D_{s,n},\|\cdot\|\opn,u) \leq {n \choose s} (1+2/u)^s \leq (en/s)^s (1+2/u)^s,
\]
so that
\[
\gamma_2({\cal A},\|\cdot\|\opn) \leq C \sqrt{\frac{s}{m}} \left(\sqrt{\log(en/s)} + \int_0^1 \sqrt{\log(1+2/u)} du\right)
\leq C' \sqrt{\frac{s \log(en/s)}{m}}.
\]
The claim follows then from Theorem \ref{thm:main:chaos}.
\end{proof}

\end{document}